\documentclass[12pt]{article}
 \usepackage[margin=1in]{geometry} 
\usepackage{amsmath,amsthm,amssymb,amsfonts}
 \usepackage{amssymb}

\usepackage[utf8]{inputenc}
\usepackage[pagewise]{lineno}

\usepackage{hyperref}
 \usepackage{csquotes}
\usepackage[utf8]{inputenc}
\usepackage[english]{babel}
\usepackage{xcolor}
\usepackage{biblatex}
\title{Global existence of solutions to the fully parabolic chemotaxis system with logistic source under nonlinear Neumann boundary condition}
\author{Minh Le }
\date{\today}

\begin{document}
\maketitle

\begin{abstract}
    We study the existence of global boundedness solutions to the fully parabolic chemotaxis systems with logistic sources, $ru- \mu u^2$, under nonlinear Neumann boundary conditions, $\frac{\partial u}{\partial \nu }=  |u|^p$ where $p >1 $ in smooth bounded domain $\Omega \subset \mathbb{R}^n$ with $n \geq 2$. A recent study by Le (2023) has shown that the logistic sources can ensure that solutions are global and bounded when $n =2$ with $p < \frac{3}{2}$ and $n=3$ with $p <\frac{7}{5}$. In this paper, we extend the previous findings by demonstrating the existence of global bounded solutions when $p< \frac{3}{2}$ in any spatial dimension $n \geq 2$.
\end{abstract}
\numberwithin{equation}{section}
\newtheorem{theorem}{Theorem}[section]
\newtheorem{lemma}[theorem]{Lemma}
\newtheorem{remark}{Remark}[section]
\newtheorem{Prop}{Proposition}[section]
\newtheorem{Def}{Definition}[section]
\newtheorem{Corollary}{Corollary}[theorem]
\allowdisplaybreaks
\section{Introduction}
In this paper, we consider the following PDEs arising from chemotaxis  
\begin{equation} \label{1.1}
    \begin{cases} 
u_{t}=  \Delta u -  \nabla \cdot (u \nabla v) +au -\mu u^2   \qquad &x\in {\Omega},\, t \in (0,T_{\rm max }), \\ 
 \tau  v_t  = \Delta v+ u -  v \qquad &x\in {\Omega},\, t \in (0,T_{\rm max }),
\end{cases} 
\end{equation}
in a smooth, convex, bounded domain $\Omega \subset \mathbb{R}^n$ with $n\geq 2$, where $a, \mu >0$ and $\tau \in \left \{ 0,1 \right \}$. The system \eqref{1.1} is complemented with the nonnegative initial conditions in $C^{2}(\bar{\Omega})$,  not identically zero:
\begin{equation} \label{1.1.2}
    u(x,0) = u_0(x), \qquad v(x,0)= v_0(x), \qquad x\in \Omega,
\end{equation}
and the nonlinear Neumann boundary conditions
\begin{equation} \label{boundary-data}
    \frac{\partial u}{\partial \nu } = |u|^p, \qquad  \frac{\partial v}{\partial \nu }  = 0, \qquad x\in \partial \Omega,\, t \in (0, T_{\rm max}),
\end{equation}
where $p>1$ and $\nu $ is the outward normal vector.\\
Under the homogeneous Neumann boundary conditions, the system \eqref{1.1} when $a=\mu =0$ used to describe the movement of cells toward chemical concentration known as chemotaxis has been first introduced in \cite{Keller}. This system in 2d has an intriguing property known as critical mass, which states that if the total mass is less than a certain value solutions are global and bounded (see \cite{Dolbeault,Dolbeault1,Mizoguchi, NSY,OY} ) while it is greater than that value, solutions blow up in finite time (see  \cite{Nagai1,Nagai2,Nagai4,Nagai3}). In higher dimensional case $n \geq 3$, this property has been proven in \cite{Winkler_2010} to fail to hold, precisely it was proven that finite-time blow-up solutions can be constructed for any given initial mass. In attempt to avoid blow-up phenomenon, the logistic sources, $au -\mu u^2$, introduced in \cite{TW} can guarantee that solutions are globally bounded when $\mu >\frac{n-2}{n}$ with $\tau =0$. The similar result for fully parabolic chemotaxis system when $\tau >0$, was later established in \cite{Winkler-2010}. Additionally, those interested in further exploration of qualitative and quantitative aspects of chemotaxis systems with logistic sources can consult the following references: \cite{Jin+Xiang, Lankeit2, LMLW, Li+Wang, LMZ, Tao+Winkler, MW2011, YCJZ, ZLBZ}.

In 1990s, nonlinear parabolic equations under the nonlinear Neumann boundary conditions has been studied in \cite{CFQ, Quittner1, Quittner}. Particularly, the authors study the global existence and finite-time blow-up solutions of the following equation:
\begin{equation} \tag{NBC}  \label{NBC}
    \begin{cases} 
U_{t}=  \Delta U-\mu U^Q  \qquad &x\in {\Omega},\, t \in (0,T_{\rm max }), \\ 
 \frac{\partial U}{\partial \nu } = U^P \qquad &x\in {\Omega},\, t \in (0,T_{\rm max }),  \\
 U(x,0) = U_0(x)  &x\in {\bar{\Omega}},
\end{cases} 
\end{equation}
where  $\Omega $ is a smooth bounded domain in $\mathbb{R}^n$, $Q,P>1$, $\mu>0$ and $U_0 \in W^{1,\infty}(\Omega)$ is a nonnegative function. It was shown that $P=\frac{Q+1}{2}$ is critical for the blow up in the following sense:
\begin{enumerate}
    \item if $P<\frac{Q+1}{2}$ then all solutions of \eqref{NBC} exist globally and are globally bounded,
    \item If $P>\frac{Q+1}{2}$ (or $P=\frac{Q+1}{2}$ and $\mu $ is sufficiently small ) then there exist initial functions $U_0$ such that the corresponding solutions of \eqref{NBC} blow-up in $L^{\infty}-$norm.
\end{enumerate}
Motivated by those works, \cite{Minh} introduces and studies the chemotaxis systems \eqref{1.1} under the nonlinear Neumann boundary conditions \eqref{boundary-data}.
The main objective of that work is to give an answer to the question: \textit{"What is the largest value $p$ so that logistic damping still avoids blow-up?"} \\
In \cite{Minh}, it was shown that solutions of the parabolic-elliptic system \eqref{1.1} when $\tau =0$  exist globally and remain bounded in time when $p< \frac{3}{2}$. However, in case $\tau =1$, the existence of globally bounded solutions was only shown for $n =2$ when $p<\frac{3}{2}$ and $n=3$ when $p< \frac{7}{5}$ with additional largeness assumption on $\mu$. In this paper, we will show that largest value $p$ so that logistic damping still excludes blow-up is $\frac{3}{2}$ for the fully parabolic system \eqref{1.1} when $\tau =1$ in any spacial dimension. This result is consistent with the results of \eqref{NBC} and \cite{Minh} with $\tau=0$, where critical power is shown to be $P= \frac{Q+1}{2}$ with $Q=2$ in our problem. To be more precise, our main result reads as follows:

\begin{theorem} \label{2dthm}
Let  $\Omega $ be a bounded, convex domain with smooth boundary $\partial \Omega \in C^\infty$ in $\mathbb{R}^n$ with $n \geq 2$. Assume that $\tau=1$, $1<p< \frac{3}{2}$ and initial data $u_0, v_0$ in Holder space $C^{2+\frac{1}{2}}(\Bar{\Omega})$ are nonnegative, not identically zero such that 
\begin{align*}
    \frac{\partial u_0}{\partial \nu } =|u_0|^{p}.
\end{align*}
The system \eqref{1.1} under boundary condition \eqref{boundary-data} possesses a unique nonnegative classical solution $(u,v) \in \left (C^{2+p-1}(\Bar{\Omega} \times [0,\infty)) \right )^2$. Moreover, $u,v>0$ in $\Omega \times(0,\infty)$ and remain bounded in the sense that
\begin{align*}
    \sup_{t >0} \left \{ \left \| u(\cdot,t) \right \|_{L^\infty(\Omega)}+ \left \| v(\cdot,t) \right \|_{W^{1,\infty}(\Omega)} \right \} <\infty.
\end{align*}
\end{theorem}

 The key idea of the proof for $\tau =1$ when $n=3$  in \cite{Minh} motivated by \cite{Winkler-2010} is to investigate the energy functional 
\begin{align}
    y(t) = \int_\Omega u^2 +\frac{1}{3} \int_\Omega  u|\nabla v|^2 + \int_\Omega |\nabla v|^4.
\end{align}
This allows us to take advantages of the term $-\mu u^2$, but also raises difficulties on the boundary terms when doing integral by parts. This method works well in $3D$ but requires $p< \frac{7}{5}$, which is far from the critical power $p= \frac{3}{2}$ in the parabolic-elliptic case. However, it cannot be utilized for higher dimensional cases when $n \geq 4$. Instead of using this idea, there is another way to deal with fully parabolic chemotaxis problems \cite{MWS} by only considering the energy functional $\int_\Omega u^p$. This eases the computational significantly since we do not need the deal with the chemical concentration function $v$ directly. This can be done thanks to the help of Lemma \ref{l1}, which enables us to bound $\Delta v$ by $u$ directly. \\
In the following sections, we will briefly recall the local-wellposedness results for system \eqref{1.1} as well as some important estimates in Section \ref{preliminaries}, and prove the main theorem in Section \ref{proof}.

\section{Preliminaries}\label{preliminaries}
Let us commence this section by stating the local existence result established in \cite{Minh}[Theorem 2.2]. 
\begin{lemma} \label{local-existence-theorem}
Assume that $\Omega $ is a open convex bounded domain in $\mathbb{R}^n$, where $n\geq 1$ with smooth boundary $\partial \Omega \in C^\infty$ and nonnegative functions $u_0,v_0$  are in $C^{2+p-1}(\bar{\Omega})$ such that 
\begin{equation}
    \frac{\partial u_0}{\partial \nu} = |u_0|^{p} \qquad \text{ on } \partial \Omega,
\end{equation}
where $p \in (1,2)$. Then there exists $T_{\rm max} \in (0, \infty]$ such that problem \eqref{1.1} under the boundary condition 
\begin{align*}
    \frac{\partial u}{\partial \nu } =|u|^{p}, \qquad \frac{\partial v}{\partial \nu}= 0,  \qquad \text{on } \partial \Omega \times (0,T_{\rm max}),
\end{align*}
admits a unique nonnegative solution $u,v$ in $C^{2+p-1,1+\frac{p-1}{2}}(\bar{\Omega}\times [0,T_{\rm max})$. Moreover, if $u_0, v_0$ are not identically zero in $\Omega$ then $u,v$ are strictly positive in $\Omega \times (0,T_{\rm max})$. If $T_{\rm max}< \infty$, then 
\begin{align}
    \limsup_{t \to T_{\rm max}} \left \{  \left \| u(\cdot,t) \right \|_{L^\infty(\Omega)} + \left \| v(\cdot,t) \right \|_{W^{1,\infty}(\Omega)} \right \} = \infty.
\end{align}
\end{lemma}
The forthcoming lemma concerned with the maximal Sobolev regularity for heat equation serves as a key estimate in the proof of the main result. Interested readers are referred to \cite{WMS}[Lemma 2.3] for details of the proof.
\begin{lemma} \label{l1}
    Assuming that $\Omega \subset \mathbb{R}^n$ with $n \geq 1$,  $ p \in (n, \infty)$, $f \in L^p \left ( (0,T)\times \Omega \right ) $, $g_0 \in W^{2,p}(\Omega)$ and $g$ is a classical solution to the following system
\begin{equation}
    \begin{cases}
     g_t = \Delta g  -  g + f &\text{in } \Omega \times (0,T), \\ 
\frac{\partial g}{\partial \nu} =  0 & \text{on }\partial \Omega \times (0,T),\\ 
 g(\cdot,0)=g_0   & \text{in } \Omega
    \end{cases}
\end{equation}
for some $T\in (0,\infty]$. There exists $C= C(n,p,\Omega)>0$ such that
\begin{align}
    \int_{t_0}^t  e^{\frac{ps}{2}} \int_\Omega |\Delta g|^p\,dx  \, ds \leq C \left ( \int_{t_0}^t e^{\frac{ps}{2}} \int_\Omega |f|^p \,dx \, ds + e^{\frac{pt_0}{2}} \left \| \Delta g(\cdot, t_0) \right \|^p_{L^p(\Omega)}\right )
\end{align}
for all $t \in (t_0, T)$ with $t_0:= \max \left \{ 1, \frac{T}{2}\right \}$.
\end{lemma}
The following lemma established in \cite{Minh}[Lemma 3.6] allows us to control boundary integrals by using the diffusion and the quadratic degradation. 
\begin{lemma} \label{Boundary-est}
If $r \geq \frac{1}{2}$, $p\in (1,\frac{3}{2})$, and $g \in C^1(\Bar{\Omega}),$ then for every $\epsilon>0$, there exists a positive constant $C=C(\epsilon,\Omega, p,r)$ such that 
\begin{equation} \label{ine-bdr}
    \int_{\partial \Omega} |g|^{p+2r-1} \leq \epsilon \int_{\Omega} |g|^{2r+1} +\epsilon \int_{\Omega} |\nabla g^r|^2 +C.
\end{equation}
\end{lemma}

\section{Proof of the main result}\label{proof}
We are now in position to prove our main result.
\begin{proof} [Proof of Theorem \ref{2dthm}]
First, we will show that $u  \in L^\infty \left ( (0, T_{\rm max}); L^r(\Omega) \right )$ for some fixed $r> n$. A direct calculation shows 
\begin{align} \label{Lr.1}
    \frac{1}{r}\frac{d}{dt}\int_{\Omega} u^{r}  &= \int_{\Omega} u^{r-1}u_t \notag\\
    &=  \int_\Omega u^{r-1} \left [ \Delta u - \nabla (u \nabla v) +au-\mu u^2 \right ] \notag \\
    &=-\frac{4(r-1)}{r^2}\int_{\Omega} |\nabla u^{\frac{r}{2}}|^2  -\frac{r-1}{r}\int_{\Omega}u^{r} \Delta v +a\int_{\Omega} u^{r} \notag \\
    &-\mu \int_{ \Omega} u^{r+1}. +\int_{\partial \Omega} u^{r+p-1}\, dS. 
\end{align}
Applying Lemma \ref{Boundary-est} with $\epsilon = \frac{r-1}{r^2}$ yields
\begin{align} \label{Lr.2}
    r \int_{\partial \Omega} u^{r+p-1}\, dS \leq \frac{r-1}{r} \int_\Omega |\nabla u^{\frac{r}{2}}|^2 + \frac{r-1}{r}  \int_\Omega u^{r+1}+c_2,
\end{align}
where $c_2=C(r)>0$. In light of Young's inequality, it follows that
\begin{align}\label{Lr.3}
    (ra+\frac{r+1}{2})\int_{\Omega}u^{r} \leq \frac{r\mu}{4} \int_{ \Omega} u^{r+1} +c_3,
\end{align}
where $c_3=C(r,a, \Omega)>0$.
By Young's inequality, we have
\begin{align}\label{Lr.4}
    -(r-1) \int_\Omega u^r \Delta v \leq \frac{r \mu}{4} \int_\Omega u^{r+1} +c_4 \int_\Omega |\Delta v|^{r+1},
\end{align}
where $c_4 = (r-1)^{r+1} \left ( \frac{2}{\mu (r+1)} \right )^r$. Collecting from \eqref{Lr.1} to \eqref{Lr.4} leads to 
\begin{align}\label{Lr.5}
    \frac{d}{dt} \int_\Omega u^r+ \frac{r+1}{2}\int_\Omega u^r \leq c_4 \int_\Omega |\Delta v|^{r+1}  -\left (\frac{r\mu }{2} -\frac{r-1}{r} \right ) \int_\Omega u^{r+1} +c_5,
\end{align}
where $c_5=c_2+c_3$. Choosing $\mu $ sufficiently large such that $\frac{r\mu}{4} > \frac{r-1}{r}$ and plugging into \eqref{Lr.5} implies
\begin{align}\label{Lr.6}
     \frac{d}{dt} \int_\Omega u^r+ \frac{r+1}{2}\int_\Omega u^r \leq c_4 \int_\Omega |\Delta v|^{r+1}  -\frac{r\mu}{2} \int_\Omega u^{r+1} +c_5.
\end{align}
Multiplying both sides of \eqref{Lr.6} by $e^{(r+1)t}$ and integrating from $t_0$ to $t$ where $t_0 = \min \left \{ 1, \frac{T_{\rm max}}{2} \right \}$, entails
\begin{align}\label{Lr.7}
    \int_\Omega u^r(\cdot,t) &\leq c_4 \int_{t_0}^t e^{-\frac{(r+1)(t-s)}{2}} \int_\Omega |\Delta v(\cdot,s)|^{r+1} - \frac{r\mu}{4} \int_{t_0}^t e^{-\frac{(r+1)(t-s)}{2}} \int_\Omega u^{r+1}(\cdot,s) \notag  \\
    &+ \int_\Omega u^r(\cdot,t_0) +\frac{2c_5}{r+1}.
\end{align}
Applying Lemma \ref{l1} to the first term of the right hand side of \eqref{Lr.7} yields
\begin{align}\label{Lr.8}
    c_4 \int_{t_0}^t e^{-\frac{(r+1)(t-s)}{2}} \int_\Omega |\Delta v(\cdot,s)|^{r+1} &\leq c_6 \int_{t_0} ^t e^{-\frac{(r+1)(t-s)}{2}}\int_\Omega u^{r+1}(\cdot, s) \notag \\
    & + c_6e^{-\frac{(r+1)(t-t_0)}{2}} \int_\Omega |\Delta v(\cdot,t_0)|^{r+1},
\end{align}
where $c_6=C(n,r, \Omega)>0$. From \eqref{Lr.7} and \eqref{Lr.8}, we obtain
\begin{align}
    \int_\Omega u^r(\cdot,t) \leq \left ( c_6 - \frac{r\mu }{4} \right )\int_{t_0} ^t  e^{-\frac{(r+1)(t-s)}{2}}\int_\Omega u^{r+1}(\cdot, s) +c_7,
\end{align}
where $c_7 = c_6e^{-\frac{(r+1)(t-t_0)}{2}} \int_\Omega |\Delta v(\cdot,t_0)|^{r+1} + \int_\Omega u^r(\cdot,t_0) +\frac{2c_5}{r+1}$. Choosing $\mu > \frac{4c_6}{r}$ yields that $ u \in L^\infty \left ( (0, T_{\rm max}); L^r(\Omega) \right )$ for some $r>n$. Applying Moser-iteration procedure established in \cite{Minh}[Lemma 5.2] for nonlinear Neumann boundary conditions implies that $L^\infty \left ( (0, T_{\rm max}); L^\infty (\Omega) \right )$. This, together with Lemma \ref{local-existence-theorem} implies that $T_{\rm max } = \infty$ and $u,v$ are in $C^{2+p-1} \left ( \bar{\Omega} \times [0,\infty) \right )$. The proof is now complete.
\end{proof}

\section*{Acknowledgments}
The author acknowledges support from the Mathematics Graduate Research Award Fellowship at Michigan State University.
\printbibliography

\end{document}